\documentclass[10pt,english]{smfart}

\RequirePackage[T1]{fontenc}
\RequirePackage{amsfonts,latexsym,amssymb}
\RequirePackage[frenchb]{babel}
\addto\extrasfrenchb{\bbl@nonfrenchitemize
\bbl@nonfrenchspacing}

\RequirePackage{mathrsfs}
\let\mathcal\mathscr

\usepackage[matrix,arrow]{xy}
\usepackage{url}
\usepackage{accents}
\usepackage{enumitem}
\usepackage{fge}
\newcommand{\moins}{\mathbin{\fgebackslash}}

\theoremstyle{plain}
\newtheorem{prop}[equation]{\propname}
\newtheorem{theo}[equation]{\theoname}

\newtheorem{coro}[equation]{\coroname}

\newtheorem{lemm}[equation]{\lemmname}
\theoremstyle{definition}
\newtheorem{defi}[equation]{\definame}
\theoremstyle{remark}
\newtheorem{rema}[equation]{\remaname}

\newcommand{\mar}[1]{\marginpar{\tiny #1}}

\let\cal\mathcal
\let\goth\mathfrak

\def\Q{{\bf Q}} \def\Z{{\bf Z}}

\def\N{{\bf N}}

\def\O{{\cal O}}
\def\G{{\cal G}} 

\def\dual{{\boldsymbol *}}

\def\epsilon{\varepsilon}

\def\Bcris{{\mathbb B}_{{\rm cris}}} 
\def\Bdr{{\mathbb B}_{{\rm dR}}}
\def\bst{{\rm B}_{{\rm st}}}
\def\ainf{{\rm A}_{{\rm inf}}}
\def\bcris{{\rm B}_{{\rm cris}}} \def\acris{{\rm A}_{{\rm cris}}}
\def\bdr{{\rm B}_{{\rm dR}}}

\def\piqp{{{\bf P}^1}}

 \def\A{{\bf A}} \def\B{{\bf B}}

\def\rg{{\rm R}\Gamma}
\def\ocirc#1{\accentset{\circ}{#1}}
\def\wotimes{\,\widehat\otimes\,}

\newcommand{\Qp}{\mathbf{Q}_p}

\newcommand{\ovk}{\overline{K} }

 \newcommand{\holim}{\operatorname{holim} }

  \newcommand{\proeet}{\operatorname{pro\acute{e}t}  }

 \newcommand{\eet}{\operatorname{\acute{e}t} }
 \newcommand{\proet}{\operatorname{pro\acute{e}t} }

 \newcommand{\Spf}{\operatorname{Spf} }

 \newcommand{\can}{ \operatorname{can} }

\newcommand{\synt}{ \operatorname{syn} }

\newcommand{\dr}{\operatorname{dR} }

 \newcommand{\kker}{\operatorname{Ker} }
 \newcommand{\crr}{\operatorname{cr} }

 \newcommand{\kr}{^{\scriptscriptstyle\bullet}}

 \newcommand{\so}{{\mathcal O}}

 \newcommand{\sa}{{\mathcal{A}}}
 
 \newcommand{\sx}{{\mathcal{X}}}
 \newcommand{\sss}{{\mathcal{S}}}
\newcommand{\sd}{{\mathcal{D}}}

   \def\B{{\bf B}}
      \def\A{{\bf A}}

\def\invlim{\mathop{\vtop{\ialign{##\crcr$\hfill{\lim}\hfil$\crcr
\noalign{\kern1pt\nointerlineskip}\leftarrowfill\crcr\noalign
{\kern -3pt}}}}\limits}
\def\dirlim{\mathop{\vtop{\ialign{##\crcr$\hfill{\lim}\hfil$\crcr
\noalign{\kern1pt\nointerlineskip}\rightarrowfill\crcr\noalign
{\kern -3pt}}}}\limits}
\def\lomapr#1{\smash{\mathop{\relbar\joinrel\longrightarrow}\limits^{#1}}}

\def\epsilon{\varepsilon}
\let\mathcal\mathscr

\begin{document}
\title[Cohomology of the affine space]
{On the cohomology of the affine space}
\author{Pierre Colmez}
\address{C.N.R.S., IMJ-PRG, Sorbonne Universit\'e, 4 place Jussieu,
75005 Paris, France}
\email{pierre.colmez@imj-prg.fr}
\author{Wies{\l}awa Nizio{\l}}
\address{CNRS, UMPA, \'Ecole Normale Sup\'erieure de Lyon, 46 all\'ee d'Italie, 69007 Lyon, France}
\email{wieslawa.niziol@ens-lyon.fr}
\begin{abstract}
We compute the $p$-adic geometric pro-\'etale cohomology of the rigid analytic  affine space (in any dimension).
This cohomology is non-zero, contrary to the \'etale cohomology, and can be described by means
of differential forms.
\end{abstract}
\setcounter{tocdepth}{3} 

\maketitle
\section*{Introduction}
Let  $K$ be a complete discrete valuation field of characteristic $0$ with perfect residue field of positive characteristic $p$.
Let  $C$ be the completion of an algebraic closure   $\overline K$ of $K$.
We denote by  $\G_K$ the absolute Galois group  of $K$ (it is also the group of continuous automorphisms of 
 $C$ that fix $K$).

For $n\geq 1$, let $\A^n_K$ be   the rigid analytic   affine space over $K$ of dimension~$n$   
and $\A^n$ be its scalar extension to $C$.
Our main result is the following  theorem.
\begin{theo}\label{aff1}
For  $r\geq 1$, we have isomorphisms of  $\G_K$-Fr\'echet spaces
$$H^r_{\proet}(\A^n,\Q_p(r))\simeq \Omega^{r-1}(\A^n)/\kker d\simeq \Omega^r(\A^n)^{d=0},$$
where $\Omega$ denotes the sheaf of differentials.
\end{theo}
\begin{rema}\label{aff2}
{\rm (i)} The $p$-adic pro-\'etale cohomology behaves in a remarkably different 
way from other (more classical)  cohomologies.  For example, for $i\geq 1$, we have :

$\bullet$ $H^i_{\rm dR}(\A^n)=H^i_{\rm HK}(\A^n)=0$, where $H\kr_{\rm HK}$ is Hyodo-Kato
cohomology (see \cite{HK} for its definition),

$\bullet$ $H^i_{\eet}(\A^n,\Q_\ell)=H^i_{\proet}(\A^n,\Q_\ell)=0$, if $\ell\neq p$,

$\bullet$ $H^i_{\eet}(\A^n,\Q_p)=0$. (Cf.~\cite{Ber} or Remark~\ref{aff4}.)\\
We listed the $\ell\neq p$ and $\ell=p$ cases of \'etale cohomology separately because,
if $\ell\neq p$, the triviality of the cohomology of $\A^n$ is a consequence of the triviality 
of the cohomology of the closed ball
  (which explains why the pro-\'etale cohomology is also trivial),
but the $p$-adic \'etale cohomology of the ball is highly nontrivial.

{\rm (ii)}  Using overconvergent syntomic cohomology allows to prove
a more general result~\cite[Th.\,1.8]{CDN3}: if $X$ is 
a  Stein space over $K$ admitting a semistable model over the ring of integers of $K$,
 there exists an exact sequence
$$0\to \Omega^{r-1}(X)/{\rm Ker}\,d\to H^r_{\proet}(X,\Q_p(r))\to
(\bst^+\wotimes H^r_{\rm HK}(X))^{N=0,\varphi=p^r}\to 0.$$
 However making syntomic cohomology overconvergent 
is technically demanding and the 
simple proof below uses special features of the geometry of the affine space. 

{\rm (iii)} Another possible approach (cf.~\cite{leB}) is to compute 
the pro-\'etale cohomology of the relative fundamental exact sequence
$0\to \Q_p(r)\to \Bcris^{\varphi=p^r}\to \Bdr/F^r\to 0$. 
\end{rema} 

 Let $\ocirc{\B}^n$ be the open unit ball of dimension $n$. An adaptation of the proof 
of Theorem \ref{aff1} shows the following result:
\begin{theo}\label{aff11}
For  $r\geq 1$, we have isomorphisms of  $\G_K$-Fr\'echet spaces
$$H^r_{\proet}(\ocirc{\B}^n,\Q_p(r))\simeq \Omega^{r-1}(\ocirc{\B}^n)/\kker d\simeq
\Omega^r(\ocirc{\B}^n)^{d=0}.$$
\end{theo}

\subsubsection*{Acknowledgements}
We would like to thank the referee for a careful reading of the manuscript
and useful suggestions for improving the exposition. 

\section{Syntomic variations}
If $r=1$, one can give an elementary proof of Theorem 1 using Kummer theory, but it does not
seem very easy to extend this kind of methods to treat the case $r\geq 2$.
Instead we are going to use syntomic methods.

Recall that the \'etale-syntomic comparison theorem \cite{Ts,CN} reduces the computation of $p$-adic \'etale cohomology to that of syntomic cohomology\footnote{
The computations in \cite{CN} are done over  $K$ (or over its finite extensions), but working directly over  $C$ 
simplifies a lot the local arguments  because there is no need to change the Frobenius and the group
$\Gamma$ of loc.\,cit.~becomes commutative 
(hence so does its  Lie algebra, which makes the arguments using Koszul complexes
a lot simpler).}.  The latter is defined as a filtered Frobenius eigenspace of absolute crystalline 
cohomology (see \cite{FM} for a gentle introduction and \cite{Ts} for a more thorough treatment) 
and can be thought of as a higher dimensional version of the Fontaine-Lafaille functor. 
Its computation reduces to a computation of cohomology of complexes built from differential forms,
and hence is often doable. 

 More precisely, 
if $\sx$ is  a quasi-compact semistable $p$-adic formal scheme over  $\so_K$, then
the  Fontaine--Messing period map \cite{FM}
   \begin{equation}
    \label{comp1}
    \alpha^{FM}: \tau_{\leq r}\rg_{\synt}(\sx_{\so_C},\Z_p(r))\to \tau_{\leq r}\rg_{\eet}(\sx_{C},\Z_p(r))
    \end{equation}
 is a $p^N$-quasi-isomorphism\footnote{It means that the kernel and 
cokernel of the induced map on cohomology are
annihilated by~$p^N$.} for a constant $N=N(r)$.
 This generalizes easily to semistable $p$-adic formal schemes over $\so_C$: the rational \'etale and pro-\'etale cohomology of such schemes are computed 
 by the syntomic complexes $\rg_{\synt}(\sx_{\so_C},\Z_p(r))_{\Q}$ and $\rg_{\synt}(\sx_{\so_C},\Q_p(r))$, respectively, where  the latter complex is defined by taking $\rg_{\synt}(\sx_{\so_C},\Z_p(r))_\Q$ locally and then glueing.

The purpose of this section is to construct a particularly simple complex that, morally, computes
the syntomic, and hence (pro-)\'etale as well, cohomology of the (canonical formal model of the) 
affine space and the open ball, but does not use a model of the whole space, only of closed balls
of increasing radii.

\smallskip
\noindent{\it Period rings}. ---
Let $C^\flat$ be the tilt of $C$ and let 
$\acris\subset\bcris^+=\acris[\frac{1}{p}]\subset\bdr^+$ be the usual Fontaine's
rings.  

Let $\theta:\bdr^+\to C$ be the canonical projection (its restriction to $\acris$ induces
a projection
$\acris\to \so_C$), and let ${\rm F}\kr_\theta\bdr^+$ be the filtration
by the powers of ${\rm Ker}\,\theta$ and ${\rm F}\kr_\theta\acris$ be the induced filtration.
For  $j\in\Z$, let  ${\rm A}_j=\acris/{\rm F}^j_{\theta}$
(hence  ${\rm A}_j=0$ for  $j\leq 0$ and ${\rm A}_1=\so_C$). 

We choose a morphism of groups  $\alpha\mapsto p^{\alpha}$ from $\Q$ to $C^\dual$
compatible with the analogous morphism on  $\Z$.  We denote by  $\tilde p^\alpha$
the element  $(p^\alpha,p^{\alpha/p},p^{\alpha/p^2},\dots)$ of $C^\flat$ and by  $[\tilde p^\alpha]$ its Teichm\"uller lift in $\acris$.

\smallskip
\noindent{\it Closed balls}. ---
For $\alpha\in \Q_+$, let $$D_\alpha=\{z=(z_1,\dots,z_n),\ v_p(z_m)\geq -\alpha,\ 
 \text{{for $1\leq m\leq n$}}\}$$ be the closed ball of valuation $-\alpha$ in $\A^n$,
and denote by  $\O(D_\alpha)$ (resp.~$\O^+(D_\alpha)$) the ring of analytic functions
 (resp.~analytic functions with integral values) on $D_\alpha$.
 We have
$$\O(D_\alpha)=C\langle p^\alpha T_1,\dots,p^\alpha T_n\rangle
\quad{\rm and}\quad
\O^+(D_\alpha)=\O_C\langle p^\alpha T_1,\dots,p^\alpha T_n\rangle.$$
Consider the lifts
$$ R_{\alpha}^+=\acris\langle [\tilde p^\alpha] T_1,\dots,[\tilde p^\alpha] T_n\rangle
\quad{\rm and}\quad
 R_{\alpha}=R_\alpha^+[\tfrac{1}{p}]$$
of $\O^+(D_\alpha)$ and $\O(D_\alpha)$, respectively. 
We extend  $\varphi$ on $\acris$ to $\varphi:R_{\alpha}\to R_{\alpha}$ by
setting  $\varphi(T_m)=T_m^p$, for  $1\leq m\leq n$.

\begin{defi}\label{AFF2}
Let $r\geq 0$. If $\alpha\in\Q_+$ and $\Lambda=R_\alpha,R_\alpha^+$, we define the complexes
$${\rm Syn}(\Lambda,r)  :=[{\rm HK}_r(\Lambda)\to {\rm DR}_r(\Lambda)],
$$
where the brackets $[\cdots]$ denote the mapping fiber, and\footnote{
The differentials are taken relative to $\acris$.}
\begin{align*}
{\rm HK}_r(\Lambda)& :=[\Omega\kr_{\Lambda}\lomapr{\varphi-p^r}\Omega\kr_{\Lambda}],\\
F^r\Omega\kr_{\Lambda}
& :=(F^r_\theta \Lambda\to F^{r-1}_\theta \Omega^1_{\Lambda}\to
F^{r-2}_\theta \Omega^2_{\Lambda}\to\cdots),\\
{\rm DR}_r(\Lambda) & :=\Omega\kr_{\Lambda}/F^r=
  (\xymatrix{\dots\to {\rm A}_{r-i}\otimes_{\acris}\Omega^i_{\Lambda}\ar[r]^-{1\otimes d_i}
&{\rm A}_{r-i-1}\otimes_{\acris}\Omega^{i+1}_{\Lambda}\to\cdots}).
\end{align*}
\end{defi}

\noindent{\it The complex ${\rm Syn}(\A^n,r)$}. ---
  The above complexes  for varying $\alpha$
are closely linked:

$\bullet$ The ring morphism $R_{0}\to R_{\alpha}$, 
$T_m\to[\tilde p^\alpha]T_m$, for  $1\leq m\leq n$,
induces an isomorphism of complexes ${\rm Syn}(R_0,r)\overset{\sim}{\to}{\rm Syn}(R_\alpha,r)$.

$\bullet$ For $\beta\geq\alpha$, the inclusion
$\iota_{\beta,\alpha}: R_{\beta}\hookrightarrow R_{\alpha}$ 
induces a morphism of complexes ${\rm Syn}(R_\beta,r){\to}{\rm Syn}(R_\alpha,r)$
thanks to the fact that $\varphi([\tilde p^s])=[\tilde p^s]^p$, for all $s\in\Q_+$.
 
\noindent (We have analogous statements replacing 
$R_\alpha$ by $R_\alpha^+$.)

The first point comes just from the fact that two closed balls are isomorphic, but the
second point, to the effect that we can find liftings of the $\O(D_\alpha)$'s with compatible
Frobenius, is a bit  of a miracle, and will simplify greatly the computation of the
syntomic cohomology of $\A^n$. In particular,
it makes it possible to define the complex
${\rm Syn}(\A^n,r):=\holim_{\alpha}{\rm Syn}(R_\alpha,r)$ and, similarly,
 ${\rm HK}_r(\A^n)$ and ${\rm DR}_r(\A^n)$. 

For $i\geq 0$ and $X=\A^n,R_\alpha,R_\alpha^+$, denote by  ${\rm HK}_r^i(X)$, 
${\rm DR}_r^i(X)$, and ${\rm Syn}^i(X,r)$ the cohomology groups of the corresponding complexes.
We have a long exact sequence:
$$\cdots\to{\rm DR}_r^{i-1}(X)\to {\rm Syn}^i(X,r)\to {\rm HK}^i_r(X)\to{\rm DR}_r^i(X)\to\cdots$$
\begin{prop}\label{aff3}
If $i\leq r$, we have natural isomorphisms:

$\bullet$ $H^i_{\eet}(D_{\alpha},\Q_p(r))\cong {\rm Syn}^i(R_\alpha,r)$, if $\alpha\in\Q_+$.

$\bullet$ $H^i_{\proet}(\A^n,\Q_p(r))\cong {\rm Syn}^i(\A^n,r)$.
\end{prop}
\begin{proof}
 Take 
$\alpha\in\Q_+$.  By the comparison isomorphism (\ref{comp1}), to prove the first claim, it suffices to show that 
the complex  ${\rm Syn}(R_\alpha,r)$ computes  the rational geometric log-syntomic
cohomology of  $\sd_\alpha:=\Spf \O^+(D_\alpha)$, the formal affine space over $\so_C$,  that is a smooth formal model of $D_\alpha$.  
To do this, recall that the latter cohomology 
is computed by the complex
$$
  \rg_{\synt}(\sd_\alpha,\Z_p(r))_\Q
  =[\rg_{\crr}(\sd_{\alpha}/\acris)_{\Q}^{\varphi=p^r}\to
  \rg_{\crr}(\sd_{\alpha}/\acris)_{\Q}/F^r],
$$
where $\acris$ is equipped with the unique  log-structure extending the canonical log-structure on $\so_C/p$. It suffices thus to show that there exists  a quasi-isomorphism 
$\rg_{\crr}(\sd_{\alpha}/\acris)_{\Q}\simeq 
\Omega\kr_{R_{\alpha}}$
that is compatible with the Frobenius\footnote{Recall that  the Frobenius on crystalline cohomology is defined via the isomorphism 
$\rg_{\crr}(\sd_{\alpha}/\acris)_{\Q}\stackrel{\sim}{\to} \rg_{\crr}((\sd_{\alpha,}/p)/\acris)_{\Q} $ from the canonical Frobenius on the second term.} and the filtration. But this is clear since $\Spf R^+_\alpha$ is a log-smooth lifting of $\sd_{\alpha}$ from $\Spf \so_{C}$ to $\Spf \acris$
that is compatible with the  Frobenius on $\acris$ and $\so^+(D_{\alpha})/p$.

To show  the second claim, we note that, for $\beta\geq \alpha$, there is a natural map (an injection) of liftings
$(R^+_{\beta}\to  \so^+(D_{\beta}))\to 
(R^+_{\alpha}\to  \so^+(D_{\alpha})). $ This allows us to use the comparison isomorphism (\ref{comp1}) to define the second quasi-isomorphism in the sequence of maps
\begin{align*}
\tau_{\leq r}\rg_{\proeet}(\A^n,\Q_p(r)) & \simeq\tau_{\leq r}\holim_k\rg_{\eet}(D_{k},\Q_p(r))\simeq\tau_{\leq r}\holim_k\rg_{\synt}(\sd_{k},\Z_p(r))_{\Q}\\
 & \simeq \tau_{\leq r}\holim_k{\rm Syn}(R_k,r)=\tau_{\leq r}{\rm Syn}(\A^n,r).
\end{align*}
Here, the first quasi-isomorphism follows from the fact that $\{D_k\}_{k\in\N}$ is an admissible affinoid covering of $\A^n$ and  the third one follows from the first claim. This finishes the proof.
\end{proof}

\section{Computation of ${\rm HK}^i_r(\A^n)$}
The group ${\rm HK}^i_r(\A^n)$ is, by construction, obtained from the ${\rm HK}^i_r(R_\alpha)$'s, but
the latter are, individually, hard to compute and quite nasty: for example, 
${\rm HK}^1_1(R_\alpha)$ is isomorphic to the quotient
of $\Qp\widehat\otimes \O(D_\alpha)^\dual$ by the sub $\Q_p$-vector 
space generated by $\O(D_\alpha)^\dual$;
hence it is an infinite dimensionnal topological $\Q_p$-vector space in which $0$ is dense.
Fortunately Lemma~\ref{AFF3} below shows that this is not a problem for the computation
of ${\rm HK}^i_r(\A^n)$.

\smallskip
For ${\bf k}=(k_1,\dots,k_n)\in\N^n$, we set  $|{\bf k}|=k_1+\cdots+k_n$ and  $T^{\bf k}=T_1^{k_1}\cdots T_n^{k_n}$.
For $1\leq j\leq n$, let $\omega_j$ be the differential form  $\frac{dT_j}{T_j}$, and  let $\partial_j$ be
differential operator defined by  $df=\sum_{j=1}^n\partial_jf\,\omega_j$.
For  ${\bf j}=\{j_1,\dots,j_i\}$, with $j_1\leq j_2\leq\cdots\leq j_i$, we set 
$\omega_{\bf j}=\omega_{j_1}\wedge \cdots\wedge \omega_{j_i}$. All elements 
 $\eta$ of $\Omega^i_{R_\alpha}$ can be written, in a unique way, in the form 
  $\sum_{|{\bf j}|=i}a_{\bf j}\omega_{\bf j}$, where  $a_{\bf j}\in \big(\prod_{j\in {\bf j}}T_j\big)R_\alpha$.

\begin{lemm}\label{aff5}
Let $M$ be a sub-$\Z_p$-module of  $\acris$ or $\O_C$.
Let $i\geq 1$ and ${\bf k}\in\N^n$.  
For  $\omega=T^{\bf k}\sum_{|{\bf j}|=i}a_{\bf j}\omega_{\bf j}$, with  $a_{\bf j}\in M$,
such that  $d\omega=0$, there exists  $\eta=T^{\bf k}\sum_{|{\bf j}|=i-1}b_{\bf j}\omega_{\bf j}$, such that 
$d\eta=\omega$ and  $b_{\bf j}\in p^{-N({\bf k})}M$, with $N({\bf k})=\inf_{j\in {\bf j}}v_p(k_j)$.
\end{lemm}
\begin{proof}
Permuting the $T_m$'s,
we can assume that  $v_p(k_1)\leq v_p(k_2)\leq \cdots\leq v_p(k_n)$; in particular, $k_1\neq 0$.
Decompose $\omega$ as $\big(\omega_1\wedge
T^{\bf k}\sum_{1\in {\bf j}}a_{\bf j}\omega_{{\bf j}\moins\{1\}}\big)+\omega'$, and
set $\eta=\frac{1}{k_1}T^{\bf k}\sum_{1\in {\bf j}}a_{\bf j}\omega_{{\bf j}\moins\{1\}}$; we have  $\omega-d\eta=
T^{\bf k}\sum_{1\notin {\bf j}}c_{\bf j}\omega_{\bf j}$ and it has a trivial differential.
But  $d(T^{\bf k}\sum_{1\notin {\bf j}}c_{\bf j}\omega_{\bf j})=k_1T^{\bf k}\sum_{1\notin {\bf j}}c_{\bf j}\omega_{\{1\}\cup {\bf j}}
+\sum_{1\notin I}c'_I\omega_I$, hence $c_{\bf j}=0$ for all  ${\bf j}$, which proves that $d\eta=\omega$ and allows us to conclude.
\end{proof}

\begin{lemm}\label{AFF3}
Let $\alpha\in\Q_+$ and let $\Lambda_\alpha=R^+_\alpha,\so^+(D_\alpha)$.  Then 
$ H^0_{\rm dR}(\Lambda_{\alpha})=\acris,\so_C$ and ${\rm HK}_r^0(R^+_{\alpha})=\acris^{\varphi=p^r}$,
the natural maps 
$$H^i_{\rm dR}(\Lambda_{\alpha+1})\to H^i_{\rm dR}(\Lambda_{\alpha}),\quad i\geq 1;\quad  {\rm HK}_r^i(R^+_{\alpha+2})\to {\rm HK}_r^i(R^+_{\alpha}),\quad i\geq2.$$
are identically zero,
and the image of the map 
$ {\rm HK}_r^1(R^+_{\alpha+2})\to {\rm HK}_r^1(R^+_{\alpha})$ is annihilated by $p^r$. 
\end{lemm}
\begin{proof}
The computation of the $H^0$'s is straightforward.
The proof for the first map is similar  (but easier) to that of the second one, so we are only going to prove the latter. Take $i\geq 2$. 
Let $(\omega^i,\omega^{i-1})$ be a representative of an element of ${\rm HK}_r^i(R^+_{\alpha+2})$.
That is to say $\omega^i\in\Omega^i_{R^+_{\alpha+2}}$,
$\omega^{i-1}\in\Omega^{i-1}_{R^+_{\alpha+2}}$,
$d\omega^i=0$ and $d\omega^{i-1}+(\varphi-p^r)\omega^i=0$.

Since $d\omega^i=0$, we deduce from Lemma~\ref{aff5}
 that there exists $\eta^{i-1}\in \Omega^{i-1}_{R^+_{\alpha+1}}$ such that
$\iota_{\alpha+2,\alpha+1}\omega^i=d \eta^{i-1}$ 
(we used here that $\frac{1}{m}[\tilde{p}]^m\in \acris$).
Let
 $\omega^{i-1}_1=
\iota_{\alpha+2,\alpha+1}\omega^{i-1}+(\varphi-p^r)\eta^{i-1}$.
Then $d\omega^{i-1}_1=
\iota_{\alpha+2,\alpha+1} d\omega^{i-1}+(\varphi-p^r)d\eta^{i-1}=0$;
hence there exists $\eta^{i-2}\in \Omega^{i-2}_{R^+_\alpha}$ such that
$\iota_{\alpha+1,\alpha}\omega^{i-1}_1=d \eta^{i-2}$.
It follows that
$\iota_{\alpha+2,\alpha}(\omega^i,\omega^{i-1})=d
(\iota_{\alpha+1,\alpha}\eta^{i-1},\eta^{i-2})$, as wanted.

  Take now $i=1$ and use the notation from the above computation. Arguing as above we show that $(\omega^1,\omega^{0})$ is in the same class as $(0,\omega^{0})$, with
  $\omega^{0}\in \acris$. But the map $\varphi-p^r: \acris\to\acris$ is $p^r$-surjective. 
This proves the last statement of the lemma. 
  \end{proof}
\begin{rema}\label{AFF3.1}
{\rm (i)}
The same arguments would prove that there exists $N:\Q_+^\dual\to \N$ such that, 
if $\beta>\alpha$ and 
$i\geq 1$, 
the images of the natural maps $H^i_{\rm dR}(R^+_{\beta})\to H^i_{\rm dR}(R^+_{\alpha})$, 
${\rm HK}_r^i(R^+_{\beta})\to {\rm HK}_r^i(R^+_{\alpha})$ are killed by $p^{N(\beta-\alpha)}$.
This is sufficient to extend Corollary~\ref{AFF4} and Lemma~\ref{AFF5} to the unit ball
$\ocirc{\bf B}^n$. 

{\rm (ii)}
Note, however, that $N(u)\to +\infty$ when $u\to 0^+$. This prevents the extension of
Lemma~\ref{AFF5} to the integral de Rham cohomology of $\ocirc{\bf B}^n$ which is
good since this integral de Rham cohomology, in degrees $1\leq i\leq n$, is far from $0$
(but its $p$-torsion is dense).

\end{rema}

\begin{coro}\label{AFF4}
If  $i\geq 1$ then  ${\rm HK}_r^i(\A^n)=0$.
\end{coro}
\begin{proof}Immediate from Lemma \ref{AFF3} and 
 the exact sequence
$$0\to {\rm R}^1\invlim_k {\rm HK}_r^{i-1}(R_{k})\to
{\rm HK}_r^i(\A^n)\to \invlim_k {\rm HK}_r^{i}(R_{k})\to
0
\qedhere
$$
\end{proof}
\section{Computation of ${\rm DR}_r^{i}(\A^n)$}

\begin{lemm}\label{AFF5}
If $1\leq i\leq r-1$ then ${\rm DR}_r^i(\A^n)\simeq \big(\Omega^{i}(\A^{{n}})/\kker d\big)(r-i-1)$,   if $i\geq r$ then ${\rm DR}_r^i(\A^n)=0$, and, if $r>0$, we have an exact sequence
$$0\to {\rm B}^+_{\rm cris}/F^r_{\theta}\to {\rm DR}_r^0(\A^n)\to  \big(\so(\A^{{n}})/C\big)(r-1)\to 0$$
\end{lemm}
\begin{proof}
We have an exact sequence 
$$0\to {\rm R}^1\invlim_k {\rm DR}_r^{i-1}(R_{k})\to
{\rm DR}_r^i(\A^n)\to \invlim_k {\rm DR}_r^{i}(R_{k})\to 0
$$
The ${\rm DR}_r^i(R_{k})$'s are the cohomology groups of the complex 
$$\xymatrix{\dots\ar[r]&{\rm A}_{r-i}\otimes_{\acris}\Omega^i_{R_{k}}\ar[r]^-{1\otimes d_i}
&{\rm A}_{r-i-1}\otimes_{\acris}\Omega^{i+1}_{R_{k}}\ar[r]&\cdots}$$
In particular, they are trivially $0$ if $i\geq r$, so assume $i\leq r-1$.
The kernel of $1\otimes d_i$ is $F^{r-i-1}_{\theta}{\rm A}_{r-i}\otimes_{\acris}\Omega^i_{R_{k}}+
{\rm A}_{r-i}\otimes_{\acris}(\Omega^i_{R_{k}})_{d=0}$ while the image
of $1\otimes d_{i-1}$ is ${\rm A}_{r-i}\otimes_{\acris}d\Omega^{i-1}_{R_{k}}$.
Since $F^{r-i-1}_{\theta}{\rm A}_{r-i}$ is an $\O_C$-module of rank~$1$ (generated by the image of
$\frac{(p-[\tilde p])^{r-i-1}}{(r-i-1)!})$, we have $F^{r-i-1}_{\theta}{\rm A}_{r-i}\otimes_{\acris}\Omega^i_{R_{k}}\simeq
\Omega^i(D_{k})(r-i-1)$, which gives us the exact sequence
$$
0\to {\rm A}_{r-i}\otimes_{\acris}H^i_{\rm dR}(R_{k})\to {\rm DR}_r^i(R_{k})\to  
\big(\Omega^i(D_{k})/\kker d\big)(r-i-1)\to 0.
$$
  For $i=0$ this gives the sequence in the lemma. 

Assume that $i\geq 1$. 
The natural map
$H^i_{\rm dR}(R_{{k}+1})\to H^i_{\rm dR}(R_{k})$ is identically zero by Lemma~\ref{AFF3}.
Hence $${\rm R}^j\invlim_k(\Omega^i(D_k)/\kker d)\simeq{\rm R}^j\invlim_k {\rm DR}_r^{i}(R_{k}), \quad j\geq 0.
$$
Now, we note that since our systems are indexed by $\N$, ${\rm R}^j\varprojlim_k$ is trivial for $j\geq 2$. 
Since $ {\rm R}^1\varprojlim_k\Omega^i(D_k)=0$, we have 
${\rm R}^1\varprojlim_k(\Omega^i(D_k)/\kker d)=0$ (and  ${\rm R}^1\varprojlim_kd\Omega^{i}=0$).
It remains to show that
$\varprojlim_k(\Omega^i(D_k)/\kker d)\simeq \Omega^i(\A^n)/\kker d$. But this amounts to showing that
${\rm R}^1\varprojlim_k\Omega^{i}(D_k)_{d=0}=0$.
This is clear for $i=0$ and for $i>0$, since the system $\{H^{i}_{\rm dR}(R_k)\}_{k\in\N}$ 
is trivial (by Lemma \ref{AFF3}), this follows from the fact that ${\rm R}^1\varprojlim_kd\Omega^{i-1}(D_k)=0$.
\end{proof}
\section{Proof of Theorem \ref{aff1} and Theorem \ref{aff11}}
\subsection{Algebraic isomorphism}
From Proposition \ref{aff3} we know that 
$ \tau_{\leq r}{\rm Syn}(\A^n,r)\simeq \tau_{\leq r}\rg_{\proeet}(\A^n,\Q_p(r))$. 
From the 
long exact sequence
$$\cdots\to {\rm DR}_r^{i-1}(\A^n)\to {\rm Syn}^i(\A^n,r)\to {\rm HK}_r^i(\A_n)\to
{\rm DR}_r^{i}(\A^n)\to\cdots$$
and
Corollary \ref{AFF4} and Lemma \ref{AFF5}, we obtain isomorphisms 
$$
\big(\Omega^{i-1}(\A^n)/\kker d\big)(r-i)\stackrel{\sim}{\to}{\rm Syn}^i(\A^n,r), \quad r\geq i \geq 2,
$$
and 
the exact sequence
$$
0\to {\rm Syn}^0(\A^n,r)\to {\rm B}^{+,\varphi=p^r}_{\rm cris}\to {\rm DR}_r^0(\A^n)\to {\rm Syn}^1(\A^n,r)\to 0,
$$
which,
using 
 the fundamental exact sequence 
$$0\to \Q_p(r)\to {\rm B}^{+,\varphi=p^r}_{\rm cris}\to  {\rm B}^+_{\rm cris}/F^r_{\theta}\to 0,  $$
 proves the first isomorphism in Theorem~\ref{aff1} (together with ${\rm Syn}^0(\A^n,r)\cong \Q_p(r)$).  
The second isomorphism is an immediate consequence of the fact that $H^i_{\dr}(\A^n)=0$.

Since an open ball is an increasing union of closed balls,
Theorem \ref{aff11} is proved by the same argument (see Remark~\ref{AFF3.1}).

  \begin{rema}\label{aff4}
{\rm (i)}
Let $j \in \N$. We note that, since $[\tilde{p}]^{p}\in p\acris$,
for every $\alpha\in\Q_+$, the maps \footnote{The subscript $j$ refers to moding out by $p^j $.} 
$\Omega^i(R^+_{\alpha+m})_j \to \Omega^i(R^+_{\alpha})_j$, $m\geq pj$,  are the zero maps for  
$i\geq 1$ and the projection on the constant term for $i=0$. 
It follows that
\begin{align*}
\holim_{k}{\rm HK}_r(R^+_{k})_j  \simeq ({\rm A}_{{\rm cris},j }\lomapr{\varphi-p^r}{\rm A}_{{\rm cris},j}),\quad 
\holim_{k}{\rm DR}_r(R^+_{k})_j   \simeq ({\rm A}_{{\rm cris}}/F^r_{\theta})_j.
\end{align*}
Computing as above we get
$
(\holim_{k,\ell }{\rm Syn}(R^+_{k},r)_j)\otimes \Q\simeq \Q_p(r).
$
Hence, by the comparison isomorphism (\ref{comp1}), 
$H^i_{\eet}(\A^n,\Q_p(r))=0$, $i\geq 1$,
which allows us to recover the result of Berkovich~\cite{Ber}.

{\rm (ii)} The above argument does not go through for the open unit ball: the integral
de Rham complex does not reduce to the constants in that case
 and $H^i_{\eet}(\ocirc{\bf B}^n,\Q_p(r))$
is an infinite dimensionnal $\Q_p$-vector space if $1\leq i\leq n$.
\end{rema}

 \subsection{Topological considerations}
 It remains to discuss topology. 
In what follows, we write $\cong$ for an isomorphism of vector spaces and $\equiv$ for
an isomorphism of topological vector spaces.

First, note that all the cohomology groups under consideration are cohomology groups
of complexes of Fr\'echet spaces (and even of finite sums of countable products of Banach spaces),
since these complexes 
can be built out of \v{C}ech complexes coming from coverings by affinoids, and
the corresponding complexes for affinoids
involve finitely many Banach spaces.  It follows that, a priori, all the groups
we are dealing with are cokernels of maps $F_1\to F_2$ between Fr\'echet spaces.
If such a group
injects continuously into a Fr\'echet space, then it is a Fr\'echet space (it is separated
hence the image of $F_1$ in $F_2$ is closed, and our space is a quotient of a Fr\'echet
space by a closed subspace), and if this
injection is a bijection then it is an isomorphism of Fr\'echet spaces by the
Open Mapping Theorem.

Now, we have the following commutative diagram:
$$\xymatrix@R=.6cm{H^r_{\proeet}(\A^n,\Q_p(r))\ar[r]\ar[d]^{\cong}
&\varprojlim_kH^r_{\eet}(D_k,\Q_p(r))\ar[d]^{\equiv}\\
{\rm Syn}^r(\A^n,r)\ar[r]^-{\cong}&\varprojlim_k{\rm Syn}^r(R_k,r)}$$
The horizontal maps are the natural maps (and are continuous), 
the bottom one being an isomorphism
by the earlier computations.
The left vertical arrow is an isomorphism by Proposition \ref{aff3} and 
 the right vertical arrow is a topological isomorphism
because the period maps (\ref{comp1}) are $p^N$-quasi-isomorphisms, with $N$ depending only on $r$.
Thus proving that $\varprojlim_k{\rm Syn}^r(R_k,r)$ is Fr\'echet would imply
that so is $H^r_{\proeet}(\A^n,\Q_p(r))$ and that
$H^r_{\proeet}(\A^n,\Q_p(r))\equiv \varprojlim_k{\rm Syn}^r(R_k,r)$.

 For that, consider the  map of distinguished triangles
   $$
 \xymatrix@R=.6cm{
  {\rm Syn}(R_k,r)\ar[r]\ar[d]^{\alpha} & {\rm HK}_r(R_k)\ar[r] \ar[d]^{\beta}&   {\rm DR}_r(R_k)\ar[d]^{\gamma}\\
\Omega^{\geq r}(D_k)[-r] \ar[r] &   \Omega\kr(D_k)\ar[r] &   \Omega^{\leq r-1}(D_k)
}
$$
in which:

$\bullet$ the top line is the definition of ${\rm Syn}(R_k,r)$, the bottom one is the obvious one,

$\bullet$ $\gamma$ is obtained by applying $\theta$ to the terms of the complex ${\rm DR}_r(R_k)$,

$\bullet$ $\beta$ is obtained by composing the natural map ${\rm HK}_r(R_k)\to \Omega\kr_{R_k}$
with $\theta$,

$\bullet$ $\alpha$ is obtained by composing the natural map
${\rm Syn}(R_k,r)\to F^r\Omega\kr_{R_k}$
with $\theta$.

All the maps  are continuous (including the boundary maps). For $r\geq 2$, taking cohomology and limits we obtain the commutative diagram
$$
\xymatrix@R=.6cm{
 \varprojlim_k{\rm DR}_r^{r-1}(R_k) \ar[r]^{\partial}_{\cong}\ar[d]^{\cong} & \varprojlim_k{\rm Syn}^r(R_k,r)\ar[d] \\
\Omega^{r-1}(\A^n)/\kker d\equiv \varprojlim_k(\Omega^{r-1}(D_k)/\kker d) \ar[r]^-{d}_-{\cong} & \varprojlim_k\Omega^r(D_k)^{d=0}\equiv \Omega^r(\A^n)^{d=0}
}
$$
The bottom map is an isomorphism because $\varprojlim_kH^r_{\dr}(D_k)\simeq H^r_{\dr}({\mathbf A}^n)=0$. 
The top map is an isomorphism because, on level $k$,  
its kernel and cokernel  are controlled by $ {\rm HK}^{r-1}_r(R_k)$ and $ {\rm HK}^{r}_r(R_k)$ 
respectively, which die in $R_{k-2}$ by Lemma \ref{AFF3}, and the left vertical map
is an isomorphism by the proof of Lemma~\ref{AFF5}.
The space $\Omega^r({\bf A}^n)$ is Fr\'echet; it follows that all other spaces are
also Fr\'echet (in particular $\varprojlim_k{\rm Syn}^r(R_k,r)$)
and that all the maps are topological isomorphisms.
This concludes the proof of Theorem~\ref{aff1} if $r\geq 2$.

 For $r=1$, the argument is similar, with $\varprojlim_k{\rm DR}_r^{r-1}(R_k)$ in the above diagram replaced by $(\varprojlim_k{\rm DR}_r^{r-1}(R_k))/C$.

 The proof for the open ball is similar.

\bigskip

\end{document}